\documentclass[11pt, reqno]{amsart}

\usepackage{amsmath,amssymb,amsthm, epsfig}
\usepackage{hyperref}
\hypersetup{colorlinks=true, linkcolor=blue, citecolor=red, linktoc=page, pdftitle={Non-local diffusion with free boundaries}, pdfauthor={D. Marcon, R. Teymurazyan}}
\usepackage{ulem}
\usepackage[mathscr]{euscript}
\usepackage[latin1]{inputenc}

\usepackage{xcolor}
\usepackage{hyperref}
\usepackage[nameinlink]{cleveref}
\hypersetup{
	colorlinks,
	linkcolor={blue!50!black},
	citecolor={blue!50!black},
	urlcolor={blue!50!black}
}

\usepackage{enumerate}
\usepackage{mathrsfs}
\usepackage[usenames,dvipsnames]{pstricks}
\usepackage{pst-grad}
\usepackage{pst-plot}

\newtheorem{theorem}{Theorem}

\newtheorem{lemma}{Lemma}
\newtheorem{proposition}{Proposition}
\newtheorem{corollary}{Corollary}
\newtheorem{remark}{Remark}

\date{}
\numberwithin{equation}{section}
\numberwithin{theorem}{section}
\numberwithin{lemma}{section}
\numberwithin{corollary}{section}
\numberwithin{remark}{section} 
\numberwithin{proposition}{section}
\numberwithin{definition}{section}

\newcommand{\dd}{\, \mathrm{d}}

\newcommand{\R}{\mathbb{R}}

\newcommand{\dist}{\operatorname{dist}}

\newcommand{\e}{\varepsilon}

\begin{document}
\title[Non-local diffusion with free boundaries]{Non-local diffusion with free boundaries}

\author[D. Marcon]{Diego Marcon}
    \address{Departamento de Matem\'{a}tica Pura e Aplicada, Universidade Federal do Rio Grande do Sul, Porto Alegre - RS, Brazil.}
\email{diego.marcon@ufrgs.br}

\author[R. Teymurazyan]{Rafayel Teymurazyan}
	\address{Applied Mathematics and Computational Sciences Program (AMCS), Computer, Electrical and Mathematical Sciences and Engineering Division (CEMSE), King Abdullah University of Science and Technology (KAUST), Thuwal, 23955-6900, Kingdom of Saudi Arabia}{}
	\email{rafayel.teymurazyan@kaust.edu.sa}

\begin{abstract}
We prove optimal regularity and derive several geometric properties for solutions of a free boundary problem with fractional diffusion. Additionally, we deduce local $C^{1,\alpha}$ regularity results for the corresponding interior and exterior free boundaries.

\bigskip


\bigskip

\end{abstract}

\maketitle

\section{Introduction}
The study of optimization problems with free boundaries has advanced significantly, in part due to the seminal work of Alt and Caffarelli, \cite{AC81}. Many subsequent studies have followed, such as \cite{BMW06, OT06, ST24, T05, T10, TT15, TU17, Y16}, to name just a few. These problems can be categorized based on their local or non-local nature. In the local setting, checking whether an equation holds at a particular point requires knowing the function's values within an arbitrarily small neighborhood of that point. In contrast, the non-local setting requires global knowledge of the function's values. Therefore, when considering long-range interaction, non-local models are more accurate. In other words, unlike local models, which can feel changes only on the boundary of the substance, non-local ones are sensitive to changes that occur far away.

\smallskip

In \cite{Y16}, the author studies an optimization problem in heat conduction with minimal temperature constraint, interior heating, and exterior insulation. The model is generated by the Laplace operator. Using several perturbation parameters, a new functional is studied, and the problem is eventually reduced to an Alt-Caffarelli type minimization problem. The rough idea is that these perturbed functions have regular enough solutions that converge to a solution of the original problem. Interestingly, there is no need to pass to the limit in one of those parameters, as the required configuration is already achieved once the parameter is small enough. This approach is later used in \cite{TU17} in the study of the problem for the infinity Laplacian operator as a limit of solutions from the divergence structured $p$-Laplace operator. 

\smallskip

In this paper, we study the non-local counterparts of these problems. Such models arise, for instance, in the study of best insulation devices and in financial mathematics as a pricing model for American options, \cite{BM24, CT04, M24}. For example, when looking for a rational price of an American option, where the prices of assets are modeled by a L\'evy process, we encounter a non-local obstacle-type problem -- the obstacle being the payoff function. 

\smallskip

Mathematically, for given a bounded domain $\Omega\subset\R^n$ with smooth boundary and a given smooth, non-negative function $\varphi:\mathbb{R}^n\rightarrow\mathbb{R}$,  which is compactly supported in $\Omega$, we look for a function $u:\R^n\rightarrow\R$ that minimizes the energy
\begin{equation}\label{P}\tag{P}
\min_{u\in\mathbb{M}}J(u),
\end{equation}
where
\begin{equation}\label{Jdef}
    J(u):=\frac{c_{n,s}}{4}\int_{\R^n}\int_{\R^n}\frac{\left|u(x)-u(y)\right|^2}{|x-y|^{n+2s}}\,\mathrm{d}x\,\mathrm{d}y,
\end{equation}
and $\mathbb{M}$ is the set of functions $u\in H^s(\R^n)$ for which 
\begin{equation}\label{problem}
\begin{cases}
&u\ge\varphi, \\
&(-\Delta)^su \ge 0 \text{ in } \Omega, \\
&(-\Delta)^su=0 \text{ in } \{u>0\}\setminus\Omega, \\
&\left| \{u>0\}\setminus\Omega\right|=\gamma.
\end{cases}
\end{equation}
Here, 
$$
(-\Delta)^su(x):=c_{n,s}\,\textrm{PV}\int_{\R^n}\frac{u(x)-u(y)}{|x-y|^{n+2s}}\,\mathrm{d}y
$$
is the fractional Laplacian and $c_{n,s}$ is a normalization constant (see, for example, \cite{PST25, S24, T24}).  In \eqref{problem}, the constant $\gamma>0$ is pre-determined, and $|E|$ denotes the $n$-dimensional Lebesgue measure of the set $E\subset\R^n$. 

To treat problem \eqref{P}, through the penalization technique, we reduce it to the problem studied in \cite{TT15}, where no lower bound is imposed, and solutions do not obey any prescribed behavior inside or outside of the domain. Unlike \cite{TU17,Y16}, the operator is non-local and does not have a divergence structure. A way to overcome this can be the localization of the problem (as in \cite{TT15}), using the celebrated Caffarelli-Silvestre extension argument, \cite{CS07} -- writing the fractional Laplacian as a ``Dirichlet to Neumann'' map. The latter, however, comes with a price of a weighted term in the functional.

In this paper, we do not use the aforementioned extension argument. Consequently, most of our conclusions for the minimizers of the penalized functional remain valid not only for the fractional Laplacian but also for non-local operators with kernels comparable to that of the fractional Laplacian \cite{CS09} (or even more general kernels as in \cite{CTU20}). This method falls short only in the reduction argument where we use \cite{TT15} due to the specific application of the Caffarelli-Silvestre extension argument to the fractional Laplacian. Nevertheless, we prove that solutions to problem \eqref{P} are locally non-degenerate and $s$-H\"older continuous, achieving optimal regularity. Additionally, we show that the exterior free boundary, that is, the set $\partial\left(\{u>0\}\right)$, has finite $(n-1)$-dimensional Hausforff measure. Furthermore, we obtain local $C^{1,\alpha}$ regularity results for the corresponding interior and exterior free boundaries.

\smallskip

The paper is organized as follows: we start in  \Cref{prelim}, with the mathematical set-up of a three-parameter penalization problem and prove the existence of its minimizers (\Cref{existence}). We also collect some known results for future reference. In \Cref{s2}, we obtain the boundedness of minimizers for the penalized problem (\Cref{p2.1}). In \Cref{s3}, we obtain uniform (in one of the parameters) estimates, which allow us to reduce the problem to the study of a two-parameter penalization functional (\Cref{c3.1}). In \Cref{s4}, we deduce uniform H\"older estimates in one of the remaining two parameters (\Cref{t4.1}) - reducing the problem to the study of a single parameter minimization problem, which is studied in \Cref{s5}. We show that when this last parameter is small enough (but fixed), then solutions of the penalized problem turn into solutions of the original problem (\Cref{t5.2}). This, in turn, implies $s$-H\"older (optimal) regularity, non-degeneracy, and positive density results (\Cref{t5.3}). We conclude the paper with local $C^{1,\alpha}$ regularity results for the corresponding interior and exterior free boundaries (\Cref{t6.2} and \Cref{t5.5}).

\section{Preliminaries and first results}\label{prelim}
In this section, we introduce a three-parameter penalization problem and derive the existence of minimizers. We also recall two known results for future reference and finish the section by representing some notations that are used throughout the paper.

\smallskip

For three parameters $\sigma, \delta,\e\in (0,1)$, we introduce the following penalized functional
\begin{equation}\label{2.1}
I_{\sigma,\delta,\varepsilon}(u):=J(u) + \int_{\R^n} g_\sigma(u-\varphi) \,\mathrm{d} x + f_\varepsilon\left(\int_{\Omega^c} h_\delta(u(x))\dd x\right),
\end{equation}
where $J(u)$ is defined by \eqref{Jdef},   
\begin{enumerate}[$(i)$]
\item the function $g_\sigma:\R\to\R$ is smooth, non-negative, decreasing, convex, and such that 
$$
g_\sigma(t)=
\begin{cases} -\frac{1}{\sigma}(t+\frac{\sigma}{2}), & t\le -\sigma, \\ \text{smooth}, &  -\sigma\le t \le 0, \\ 0, &   t\ge0;
\end{cases}	
$$
\item the function $h_\delta:\R\rightarrow\R$  is continuous and vanishes on $(-\infty,0]$, it is linear on~$[0,\delta]$, and it equals 1 on $[\delta,+\infty)$;
\item the function $f_\varepsilon:\R\rightarrow\R$ is given by 
$$ 
f_\e(t)=\begin{cases} \frac{1}{\e}(t-\gamma) & \text{for }t\ge\gamma, \\ \varepsilon(t-\gamma) & \text{for }t\leq\gamma. \end{cases}	
$$	
\end{enumerate}
The term $g_\sigma(v-\varphi)$ penalizes functions that do not lie above $\varphi$, the term $h_\delta$ regularizes the map $u\mapsto|\{u>0\}\setminus\Omega|$, and $f_\varepsilon$ penalizes functions whose positivity set does not have the desired volume $\gamma$  (see \cite{TU17,Y16}).

Throughout the paper, we assume that $\varphi\in C_0^\infty(\Omega)$ is such that
\begin{equation}\label{condition on varphi}
    |(-\Delta)^s\varphi|\le C_\varphi,
\end{equation}
where $C_\varphi$ is a constant depending only on $\varphi$, $s$ and $n$.
\begin{proposition}\label{existence}
    The functional $I_{\sigma,\delta,\varepsilon}:H^s(\R^n) \longrightarrow \R$, defined by \eqref{2.1}, has a minimizer.
\end{proposition}
\begin{proof}
    Observe that
	$$
	I_{\sigma,\delta,\e}(\varphi)\le J(\varphi)=:M<\infty,
	$$ where the constant $M$ is independent of $\sigma,\delta,\e$. Since $I_{\sigma,\delta,\varepsilon}\ge -\e \gamma$, there is a minimizing sequence $\{u_k\}$ such that $I_{\sigma,\delta,\varepsilon}(u_k)\le M+1$, for $k$ large enough. The sequence $\{u_k\}$ is bounded in $H^s(\R^n)$. Thus, we can extract a weakly converging subsequence in $H^s(\R^n)$, which we still denote by $\{u_k\}$. If $u$ is the weak limit, by the lower semicontinuity of $J$, we have
	$$
	J(u)\le\liminf_{k\to\infty} J(u_k).
        $$
    To pass to the limit in the other terms of $I_{\sigma,\delta,\varepsilon}$, we use the fractional analog of the Rellich-Kondrachov Theorem, \cite[Theorem 7.1]{DPV12}. The latter implies $H^s\subset\subset L^2$ on bounded domains. Since $u_k\to u$ weakly in $H^s(\R^n)$ and $\R^n=\displaystyle\cup_{i=1}^\infty B_i$, where $B_i$ is the ball of radius $i$ centered at the origin, we have, up to a subsequence, $u_k\to u$ strongly in $L^2(B_i)$, for each $i\in\mathbb{N}$, and hence $u_k\to u$ a.e. on $B_i$. By a diagonal argument, we conclude that up to a subsequence, $u_k\to u$ a.e. in $\R^n$, as $k\to\infty$. Since $g_\sigma$ is smooth and non-negative, using Fatou's lemma, we get
     $$
     \int_{\R^n}g_\sigma(u-\varphi)\dd x\le\liminf_{k\to\infty}\int_{\R^n}g_\sigma(u_k-\varphi)\dd x.
     $$        
            Similarly, since $h_\delta(u_k)\to h_\delta(u)$ a.e. in $\R^n$, and $f_{\varepsilon}$ is Lipschitz continuous and increasing, we deduce      
        $$
        f_\varepsilon\left(\int_{\Omega^c} h_\delta(u(x))\dd x\right)\le\liminf_{k\to\infty}f_\varepsilon\left(\int_{\Omega^c}           h_\delta(u_k(x))\dd x\right).                  
	$$
	Hence, 
	$$
	I_{\sigma,\delta,\e}(u) \le \liminf_{k\to\infty} I_{\sigma,\delta,\e}(u_k)= \inf_{w\in H^s(\R^n)}  I_{\sigma,\delta,\e}(w).
	$$
    Therefore, $u$ is a minimizer of $I_{\sigma,\delta,\e}$.
\end{proof}
Next, we recall two theorems from \cite{TT15}. For the proof of the following result, we refer the reader to \cite[Theorem 3.1]{TT15}.
\begin{theorem}\label{TTfb}
    If $u$ is a minimizer of 
    \begin{equation}\label{TTproblem}
    J(u)+f_\varepsilon(|\{u>0\}\setminus\Omega|),    
    \end{equation}
    where $J(u)$ is defined by \eqref{Jdef}, then
	\begin{itemize}
		\item $\mathcal{H}^{n-1}(\mathcal{K}\cap\partial\{u>0\}\cap\mathbb{R}^n)<\infty$, for every compact set $\mathcal{K}\subset\Omega$.
		\item The reduced free boundary $\partial^*\{u>0\}\cap\mathbb{R}^n$ is locally a $C^{1,\beta}$ surface, for some $\beta\in(0,1)$.
	\end{itemize}
\end{theorem}
The proof of the next theorem can be found in \cite[Theorem 5.1]{TT15}.
\begin{theorem}\label{TT}
    If $\varepsilon>0$ is small enough, then any minimizer $u$ of \eqref{TTproblem}  satisfies $|\{u>0\}\setminus\Omega|=\gamma$. In particular, it is a minimizer of $J(u)$.   
\end{theorem}
For future reference, we also recall the following result from \cite[Proposition 2.9]{S07} (see also \cite{PST25, T24}).
\begin{proposition}\label{prop:silv-regularity}
	Let $u\in L^\infty(\R^n)$ be such that $w : = (-\Delta)^{s} u \in L^\infty(\R^n)$.
	\begin{itemize}
		\item If $2 s \le 1$, then, for any $\theta< 2 s$, 
		\[
		\|u\|_{C^{0,\theta}}\le C \left(\|u\|_\infty+\|w\|_\infty\right).
		\]
		\item If $2 s > 1$, then, for any $\theta < 2 s -1$,
		\[
		\|u\|_{C^{1,\theta}}\le C \left(\|u\|_\infty+\|w\|_\infty\right).
		\]
	\end{itemize} In both cases, the constant $C>0$ depends only on $n$, $\theta$, and $s$.
\end{proposition}
\begin{remark}\label{r2.3}
	\Cref{prop:silv-regularity} holds in a neighborhood of any point $x_0\in\R^n$ (see the proof of \cite[Proposition 2.9]{S07}).
\end{remark}

\subsection{Notations}
$H^s(\R^n)$ is the fractional Sobolev space of order $ s\in(0,1)$ with norm 
$$
\|u\|_{H^s(\R^n)}:= \left(\|u\|_{L^2(\R^n)}^2+\int_{\R^n}\int_{\R^n}\frac{ \left|u(x)-u(y)\right|^2}{|x-y|^{n+2s}}\,\mathrm{d}x\,\mathrm{d}y\right)^{\frac12}.
$$
$B_r(x_0)$ is the ball of radius $r$ centered at $x_0$, and $B_r:=B_r(0)$, $|E|$ is the $n$-dimensional Lebesgue measure of the set $E$. We also use the following notations:
$$
\|u\|_{L^\infty(\Omega)}:=\sup_{\Omega}|u|\,\,\,\mbox{ and }\,\,\,\|u\|_\infty:=\|u\|_{L^\infty(\R^n)}.
$$
For a multi-index $\beta=(\beta_1,\beta_2,\ldots,\beta_n)$, we use $|\beta|:=\beta_1+\beta_2+\ldots+\beta_n$. For $\alpha\in(0,1)$ the H\"older semi-norm is defined by:
$$
[u]_{C^{0,\alpha}(\Omega)}:=\sup_{x\neq y}\frac{|u(x)-u(y)|}{|x-y|^\alpha},
$$
$$
[u]_{C^{1,\alpha}(\Omega)}:=\max_{|\beta|=1}[D^\beta u]_{C^{0,\alpha}(\Omega)},
$$
where $D^\beta u:=\partial_{x_1}^{\beta_1}\ldots\partial_{x_n}^{\beta_n}u$. Also for $k=0,1$,
$$
\|u\|_{C^{k,\alpha}(\Omega)}:=\sum_{|\beta|\le k}\|D^\beta u\|_{L^\infty(\Omega)}+\sum_{|\beta|=k}[D^\beta u]_{C^{0,\alpha}(\Omega)}
$$
and
$$
\|u\|_{C^{k,\alpha}}:=\|u\|_{C^{k,\alpha}(\R^n)}.
$$

\smallskip

\section{\texorpdfstring{$L^\infty$}{L-infinity}-bounds}\label{s2}

In this section, we show that minimizes of $I_{\sigma,\delta,\varepsilon}$ are bounded.
\begin{proposition}\label{p2.1}
If $u$ is a minimizer of $I_{\sigma,\delta,\varepsilon}$, then 
\begin{equation}\label{2.2}
0\le u \le \| \varphi\|_\infty.
\end{equation}
\end{proposition} 
\begin{proof}	
	To prove the lower bound of \eqref{2.2}, we define a competing function $v\in H^s(\R^n)$ by
	\begin{equation*}
	v:= \begin{cases}
	u, & \text{ if }u\ge0,\\
	\frac{u}{2}, & \text{ if } u<0.
	\end{cases}
	\end{equation*}
	Clearly, $v\ge u$, hence, $v-\varphi\ge u-\varphi$, and since $g_{\sigma}$ is decreasing,
	$$
	\int_{\R^n} g_\sigma(v-\varphi)\,\mathrm{d}x\le\int_{\R^n}  g_\sigma(u-\varphi)\,\mathrm{d}x.
	$$
	Observe additionally that  
	$$
	h_\delta(v)=h_\delta(u).
	$$ 
	The latter is a consequence of the fact that $h_\delta$ vanishes on $(-\infty,0]$ and $v=u$ for $u\ge0$. Since $u$ is a minimizer, we estimate
	\begin{align*}
	0 &\le I_{\sigma,\delta,\varepsilon}(v)-I_{\sigma,\delta,\varepsilon}(u)\leq J(v)-J(u)\\
	&\le\int_{\{u<0\}}\int_{\{u<0\}}\frac{|v(x)-v(y)|^2-|u(x)-u(y)|^2}{|x-y|^{n+2s}}\,\mathrm{d}x\,\mathrm{d}y\\
	&=-\frac{3}{4}\int_{\{u<0\}}\int_{\{u<0\}}\frac{|u(x)-u(y)|^2}{|x-y|^{n+2s}}\,\mathrm{d}x\,\mathrm{d}y.
	\end{align*}
	Hence, one must have $|\{u<0\}|=0$, that is, $u\ge0$ almost everywhere.	
	
	To see that $u\le\|\varphi\|_\infty$, it is enough to take as a competing function 
	\[
	w:=\begin{cases}
	u, &  \text{ if } u\le\|\varphi\|_\infty,\\ 
	\frac12\left(u+\|\varphi\|_\infty\right), & \text{ if }u>\|\varphi\|_\infty.
	\end{cases}.
	\] 
	Observe that by the definition of $g_\sigma$ and $w$, one has
	$$
	\int_{\R^n} g_\sigma(w-\varphi)\,\mathrm{d}x=\int_{\R^n} g_\sigma(u-\varphi)\,\mathrm{d}x,
	$$
	since $g_\sigma(u-\varphi)=0$, when $u>\|\varphi\|_\infty$ and
	\[
	g_\sigma(w-\varphi)=\begin{cases}
		g_\sigma(u-\varphi), & \text{ if } u\le\|\varphi\|_\infty,\\ 
		0, & \text{ if }u>\|\varphi\|_\infty.
	\end{cases}.
	\] 
	Similarly, $h_\delta(w)=h_\delta(u)$, once $\delta>0$ is small enough: $\delta\le\|\varphi\|_\infty$. Again, as $u$ is a minimizer, one has
    \begin{align*}
	0&\le I_{\sigma,\delta,\varepsilon}(w)-I_{\sigma,\delta,\varepsilon}(u)= J(w)-J(u)\\
	&\le-\frac34\int_{\{u>\|\varphi\|_\infty\}}\int_{\{u>\|\varphi\|_\infty\}}\frac{|u(x)-u(y)|^2}{|x-y|^{n+2s}}\,\mathrm{d}x\,\mathrm{d}y.
	\end{align*}
	Thus, $u\le\|\varphi\|_\infty$ almost everywhere.
\end{proof}
\begin{remark}\label{r2.1}
    Recall that (see, for example, \cite[Lemma 12.13]{S24})
    for any $u,v\in H^s(\R^n)$, one has
    $$
    \int_{\R^n}(-\Delta)^{s/2}u(-\Delta)^{s/2}v\,dx=\frac{c_{n,s}}{2}\int_{\R^n}\int_{\R^n}\frac{(u(x)-u(y))(v(x)-v(y))}{|x-y|^{n+2s}}\,dx\,dy,
    $$
    and a minimizer $u$ of $I_{\sigma,\delta,\varepsilon}$ satisfies the following Euler-Lagrange equation
    \begin{equation}\label{EL}
    (-\Delta)^su+g'_\sigma(u-\varphi)+f'_\varepsilon\left(\int_{\Omega^c} h_\delta(u)\,\mathrm{d}x\right)h'_\delta(u) \chi_{\Omega^c}=0,
    \end{equation} 
    meaning,     
    \begin{equation*}
        \begin{split}
            &\int_{\R^n}(-\Delta)^{s/2}u(-\Delta)^{s/2}v\,dx+\int_{\R^n}g'_\sigma(u-\varphi)v\,dx\\
            &+\int_{\R^n}f'_\varepsilon\left(\int_{\Omega^c} h_\delta(u)\,\mathrm{d}x\right)h'_\delta(u)\chi_{\Omega^c}v\,dx=0,\,\,\,\forall v\in H^s{(\R^n)}.
        \end{split}
    \end{equation*}
\end{remark}
\begin{remark}\label{new remark}
    If $u$ is smooth enough, then (see, for example, \cite{S24})
    $$
    \int_{\R^n}(-\Delta)^{s/2}u(-\Delta)^{s/2}v\,dx=\int_{\R^n}(-\Delta)^suv\,dx,\,\,\,\forall v\in H^s(\R^n).
    $$
    Otherwise, the identity makes sense provided the right-hand side is interpreted as the duality pairing between $H^s$ and $H^{-s}$.    
\end{remark}

\begin{remark}\label{r2.2}
	Observe that $\partial\{u>0\}\subseteq\Omega^c$. Indeed, as $u\geq0$ in $\R^n$, if $x_0\in\Omega$ and $u(x_0)=0$, then $(-\Delta)^su(x_0)<0$, unless $u$ is identically zero in $\R^n$; therefore, $u$ is not an admissible function. Hence $u>0$ in $\Omega$.
\end{remark}

\smallskip

\section{Uniform estimates}\label{s3}
In this section, we prove estimates for minimizers of $I_{\sigma,\delta,\varepsilon}$ that are uniform in the parameter $\sigma$. These allow us to pass to the limit as $\sigma\to0$.

\begin{lemma}\label{l3.1}
	If $\varphi\in C_0^\infty(\Omega)$ is such that \eqref{condition on varphi} holds, and
    $u_{\sigma, \delta,\varepsilon}$ is a minimizer of $I_{\sigma,\delta, \e}$, then  	
	\begin{equation}\label{3.1}
	\|g'_\sigma(u_{\sigma, \delta,\varepsilon}-\varphi)\|_\infty\le C_\varphi,
	\end{equation} 
	where $C_\varphi>0$ is as in \eqref{condition on varphi}, and is independent of $\delta$, $\sigma$, and $\varepsilon$.
\end{lemma}
\begin{proof}
This follows from the Euler-Lagrange equation. More precisely, if $u=u_{\sigma,\delta,\varepsilon}$ is a minimizer of $I_{\sigma,\delta,\varepsilon}$, $\tilde{u}:=u-\varphi$, since $\varphi$ is supported in $\Omega$, then $u=\tilde{u}$ in $\Omega^c$,  and \Cref{r2.1} for every $v\in H^s(\R^n)$ gives
\begin{equation*}\label{3.3}
        \begin{split}
            &\int_{\R^n}(-\Delta)^{s/2}\tilde{u}(-\Delta)^{s/2}v\,dx+\int_{\R^n}(-\Delta)^{s/2}\varphi(-\Delta)^{s/2}v\,dx\\
            &+\int_{\R^n}g'_\sigma(\tilde{u})v\,dx+\int_{\R^n}f'_\varepsilon\left(\int_{\Omega^c} h_\delta(u)\,\mathrm{d}x\right)h'_\delta(\tilde{u})\chi_{\Omega^c}v\,dx=0.
        \end{split}   
\end{equation*}
As $g_\sigma'$ is a bounded smooth function, and $\tilde{u}$ is bounded, for any $k\in\mathbb{N}$, we can take $[g'_\sigma(\tilde{u})]^k$ as a test function. Indeed, since $u\in H^s(\R^n)$, and
$$
\big|[g'_\sigma(\tilde{u}(x))]^k-[g'_\sigma(\tilde{u}(y))]^k\big|\le\|([g'_\sigma]^k)'\|_{L^\infty}|\tilde{u}(x)-\tilde{u}(y)|,
$$
then also $[g'_\sigma(\tilde{u})]^k\in H^s(\R^n)$.
Thus,
\begin{equation*}\label{n4.2}
    \begin{split}
            &\int_{\R^n}(-\Delta)^{s/2}\tilde{u}(-\Delta)^{s/2}[g'_\sigma(\tilde{u}(x))]^k\,dx+\int_{\R^n}(-\Delta)^{s/2}\varphi(-\Delta)^{s/2}[g'_\sigma(\tilde{u}(x))]^k\,dx\\
            &+\int_{\R^n}[g'_\sigma(\tilde{u}(x))]^{k+1}\,dx+\int_{\R^n}f'_\varepsilon\left(\int_{\Omega^c} h_\delta(u)\,\mathrm{d}x\right)h'_\delta(\tilde{u})\chi_{\Omega^c}[g'_\sigma(\tilde{u}(x))]^k\,dx=0.
        \end{split}
\end{equation*}
Note that since $u\ge0=\varphi$ outside of $\Omega$, then $g'_\sigma(\tilde{u})$ is supported in $\Omega$. Hence, the above integrals are all, in fact, only over $\Omega$. Observe also that the last integral on the left-hand side is zero, as the integrand involves a product of functions supported in $\Omega$ and in $\Omega^c$. Furthermore, as $g_\sigma'$ is increasing, taking $k$ odd and recalling  \Cref{r2.1}, we deduce that the first term in the left-hand side is non-negative. Thus,
\begin{equation*}
    \begin{split}
        \int_{\Omega}[g'_\sigma(\tilde{u}(x))]^{k+1}\,dx&\le-\int_{\Omega}(-\Delta)^{s/2}\varphi(-\Delta)^{s/2}[g'_\sigma(\tilde{u}(x))]^k\,dx\\
        &=-\int_{\Omega}(-\Delta)^s\varphi[g_\sigma'(\tilde{u})]^k\,dx,
        \end{split}
\end{equation*}
where the equality follows from the regularity of $\varphi$ (recall \Cref{new remark}). Since $k$ is odd, recalling again that $g_\sigma'\le0$, we have
\begin{equation*}    
        \int_\Omega |g'_\sigma( \tilde{u})|^{k+1}\le\int_\Omega|(-\Delta)^s\varphi||g'_\sigma( \tilde{u})|^{k}\,dx.
\end{equation*}
Applying the H\"older inequality, we get
$$
\int_\Omega|g'_\sigma( \tilde{u})|^{k+1}\le \left[\int_\Omega|(-\Delta)^s\varphi|^{k+1}\,dx\right]^{\frac{1}{k+1}}
\left[\int_\Omega|g'_\sigma(\tilde{u})|^{k+1}\,dx\right]^{\frac{k}{k+1}}.
$$ 
Therefore,
\begin{equation}\label{3.5}
	\|g'_\sigma(\tilde{u})\|_{L^{k+1}(\Omega)}
	\le C_\varphi|\Omega|^{\frac{1}{k+1}}.
\end{equation}
Since $g'_\sigma(\tilde{u})$ is a smooth function supported in $\Omega$, we can let $k\rightarrow+\infty$ in \eqref{3.5}, obtaining \eqref{3.1}.
\end{proof}
As a consequence of \Cref{l3.1}, we have the following result.
 \begin{theorem}\label{t3.1}
	Let $\varphi\in C_0^\infty(\Omega)$ be such that \eqref{condition on varphi} holds, and $u_{\sigma,\delta,\varepsilon}$ be a minimizer of $I_{\sigma,\delta,\varepsilon}$.
	\begin{itemize}
		\item If $2 s \le 1$, then, for any $\theta<2s$, 
		\[
		\|u_{\sigma,\delta,\varepsilon}\|_{C^{0,\theta}}\le C \left(C_\varphi
		+\frac{1}{\varepsilon\delta}\right).
		\]
		
		\item If $2s>1$, then, for any $\theta<2s-1$,
		\[
		\|u_{\sigma,\delta,\varepsilon}\|_{C^{1,\theta}}\le C\left(C_\varphi
		+\frac{1}{\varepsilon\delta}\right).
		\]
	\end{itemize} In both cases, the constant $C>0$ depends only on $n$, $\theta$, $\lambda$ and $s$.
\end{theorem}

\begin{proof}
    Recalling \Cref{p2.1}, and the fact that 
    $$
    \left|f^\prime_\varepsilon\right|\le\frac{1}{\varepsilon}\,\,\textrm{ and }\,\left|h^\prime_\delta\right|\le\frac{1}{\delta},
    $$
	and using \eqref{3.1}, the first term in \eqref{EL} can be estimated uniformly in $\sigma$, as it identifies with a bounded function, that is, 
	\[
	\left|(-\Delta)^su_{\sigma,\delta,\varepsilon}\right|\le C_\varphi
	+\frac{1}{\varepsilon\delta}.
	\] 
    The latter follows from the fact that weak solutions are also distributional solutions.
    By \Cref{prop:silv-regularity}, if $2s\le1$ and $\theta<2s$, one has
	\[
	\|u_{\sigma, \delta,\varepsilon}\|_{C^{0,\theta}} \le C \left( \|u_{\sigma, \delta,\varepsilon}\|_\infty + C_\varphi
	+\frac{1}{\varepsilon\delta}\right).
	\] 
	Taking into account \eqref{2.2}, we obtain the first part of the theorem. Similarly, the second part of the theorem holds as well. 
\end{proof}

\begin{corollary}\label{c3.1}
	Up to a subsequence, as $\sigma\rightarrow0$, the function $u_{\sigma,\delta,\varepsilon}$ converges to a function $u_{\delta,\varepsilon}$ locally uniformly in $C^{\alpha}(\R^{n})$, for any $0<\alpha<2s$, and weakly in $H^s(\R^n)$. Moreover, $u_{\delta,\varepsilon}\ge\varphi$.
\end{corollary}
\begin{proof}
    Observe that thanks to \Cref{p2.1} and \Cref{new proposition} below, $u_{\sigma,\delta,\varepsilon}$ is uniformly bounded in $H^s(\R^n)$. The convergence then follows immediately from \Cref{t3.1}, and the Arzel\`a-Ascoli Theorem. To show that $u_{\delta,\varepsilon}\ge\varphi$, take any $c>0$ and any compact set $\mathbb{K}\subset\R^{n}$. Then 
	\begin{equation}\label{3.6}
	\{u_{\delta,\varepsilon}-\varphi<-c\} \cap\mathbb{K}\subset\{u_{\sigma,\delta,\varepsilon}-\varphi<-c/2\}
	\cap\mathbb{K}
	\end{equation}
	for sufficiently small $\sigma>0$. 
	On the other hand, by the construction of $g_\sigma$ and inequality $I_{\sigma,\delta,\e }(u_{\sigma,\delta,\e })\le M$ (see the proof of \Cref{existence}), we have
	$$
	\frac{c}{2\sigma}
	\left|\{u_{\sigma,\delta,\e }-\varphi<-c/2\}\cap\mathbb{K}\right|\le\int_{\R^{n}} g_\sigma(u_{\sigma,\delta,\e }-\varphi)\le M<\infty.
	$$
	This, together with \eqref{3.6}, yields $|\{u_{\delta,\e }-\varphi<-c\}\cap\mathbb{K}|=0$, since otherwise we would have a contradiction in the last inequality once $\sigma>0$ is small enough. As the number  $c>0$ and the compact $\mathbb{K}\subset\R^{n}$ are arbitrary, we conclude that $u_{\delta,\varepsilon}\ge\varphi$.
\end{proof}

\smallskip

\section{H\"older regularity of minimizers}\label{s4}
The aim of this section is to pass to the limit as $\delta\to0$, and derive uniform H\"older estimates for minimizers. First, we show that minimizers can only have a finite measure of positivity set outside of $\Omega$. 
\begin{proposition}\label{new proposition}
    If $u$ is a minimizer of $I_{\sigma,\delta,\varepsilon}$, then $|\{u>0\}\cap\Omega^c|<\infty$.
\end{proposition}
\begin{proof}
    As in the proof of \Cref{existence}, there exists a constant $M>0$, independent of $\sigma$, $\delta$ and $\varepsilon$, such that
  $$  f_\varepsilon\left(\int_{\Omega^c}h_\delta(u)\,dx\right)<M.
  $$
  We claim that 
  $$
  T:=\int_{\Omega^c}h_\delta(u)\,dx\le\gamma+M\varepsilon.
  $$
  Indeed, otherwise for some $t>0$, one has $T>\gamma+M\varepsilon+t$, therefore,
  $$
  f_\varepsilon(T)=\frac{1}{\varepsilon}(T-\gamma)>\frac{1}{\varepsilon}(M\varepsilon+t)>M,
  $$
  which is a contradiction. Thus, 
  $$
  \int_{\Omega^c}h_\delta(u)\,dx\le\gamma+M\varepsilon.
  $$
  Recalling the definition of $h_\delta$, we have 
  $$
  |\{u\ge\delta\}\cap\Omega^c|\le
  \int_{\Omega^c}h_\delta(u)\,dx\le\gamma+M\varepsilon.
  $$
  Passing to the limit, as $\delta\to0$, we obtain  
  $$
  |\{u>0\}\cap\Omega^c|\le\gamma+M\varepsilon,
  $$
  and the result follows.
\end{proof}

\begin{lemma}\label{l4.1}
	If $w\in H^s(\R^n)$, $w\geq\varphi$, where $\varphi\in C_0^\infty(\Omega)$ satisfies \eqref{condition on varphi}, and $u_{\delta,\e}$ is as in \Cref{c3.1}, then
	\begin{equation}\label{4.1}
	\begin{split}
	J(w)&+\int_{\R^n}\int_{\R^n}\frac{[w(x)-w(y)][u_{\delta,\varepsilon}(y)-u_{\delta,\varepsilon}(x)]}{|x-y|^{n+2s}}\,\mathrm{d}x\,\mathrm{d}y \\
	&+f_\e'\left(\int_{\Omega^c}h_\delta(u_{\delta,\e})\right)
	\int_{\Omega^c}h_\delta'(u_{\delta,\e})(w-u_{\delta,\e})\geq0.
	\end{split}
	\end{equation}
\end{lemma}
\begin{proof}
	Since $u_{\sigma,\delta,\e}$ is a minimizer of $I_{\sigma,\delta,\e}$, the function
	$$
	F(t) := I_{\sigma,\delta,\e}(u_{\sigma,\delta,\e}+t(w-u_{\sigma,\delta,\e})),\,\,\,t\ge0
	$$
	has a minimum at $t=0$ and so $F'(0) \ge 0$. Thus,
	\[
	\begin{split}
	&\int_{\R^n}\int_{\R^n}\frac{\left(u_{\sigma,\delta,\varepsilon}(x)-u_{\sigma,\delta,\varepsilon}(y)\right)\left[w(x)-w(y)\right]}{|x-y|^{n+2s}}\,\mathrm{d}x\,\mathrm{d}y\\
	&-\int_{\R^n}\int_{\R^n}\frac{(u_{\sigma,\delta,\varepsilon}(x)-u_{\sigma,\delta,\varepsilon}(y))^2}{|x-y|^{n+2s}}\,\mathrm{d}x\,\mathrm{d}y\\
	&+\int_{\R^n}g_\sigma'(u_{\sigma,\delta,\e}-\varphi)(w-u_{\sigma,\delta,\varepsilon})\,dx\\
	&+f_\e'\left(\int_{\Omega^c}h_\delta(u_{\sigma,\delta,\e})\right)
	\int_{\Omega^c}h_\delta'(u_{\sigma,\delta,\e})(w-u_{\sigma,\delta,\e})\geq0,
	\end{split}
	\]
	which, by the monotonicity of $g_\sigma'$ (recall that $g_\sigma$ is convex) and the elementary inequality $A(B-A)\leq B(B-A)$ for any numbers $A$ and $B$, yields:
	\begin{equation}\label{4.2}
	\begin{split}
	J(w)&+\int_{\R^n}\int_{\R^n}\frac{[w(x)-w(y)][u_{\sigma,\delta,\varepsilon}(y)-u_{\sigma,\delta,\varepsilon}(x)]}{|x-y|^{n+2s}}\,\mathrm{d}x\,\mathrm{d}y\\
	&+\int_{\R^n}g_\sigma'(w-\varphi)(w-u_{\sigma,\delta,\varepsilon})\,dx\\
	&+	f_\e'\left(\int_{\Omega^c}h_\delta(u_{\sigma,\delta,\e})\right)
	\int_{\Omega^c}h_\delta'(u_{\sigma,\delta,\e})(w-u_{\sigma,\delta,\e})\geq0.
	\end{split}
	\end{equation}
	Since $w\geq\varphi$, then $g_\sigma'(w-\varphi)=0$. We can pass to the limit, as $\sigma\to0$, in the last term of \eqref{4.2}, as in \cite[Proof of Lemma 4.1]{TU17,Y16}. For the sake of completeness, we bring it here. Since
  \begin{eqnarray*}
    &&\left|\int_{\Omega^c}h_\delta'(u_{\sigma,\delta,\varepsilon})(w-u_{\sigma,\delta,\varepsilon})-\int_{\Omega^c}h_\delta'(u_{\delta,\varepsilon})(w-u_{\delta,\varepsilon})\right|\nonumber\\
    &\leq& \left|\int_{\Omega^c}h_\delta'(u_{\sigma,\delta,\varepsilon})(u_{\sigma,\delta,\varepsilon}-u_{\delta,\varepsilon})\right|
    +\left|\int_{\Omega^c}(h_\delta'(u_{\sigma,\delta,\varepsilon})-h_\delta'(u_{\delta,\varepsilon}))(w-u_{\delta,\varepsilon})\right|\nonumber \\
    &\leq& C(\delta)\|u_{\sigma,\delta,\varepsilon}-u_{\delta,\varepsilon}\|_{L^2}+\left|\int_{\Omega^c}(h_\delta'(u_{\sigma,\delta,\varepsilon})-h_\delta'(u_{\delta,\varepsilon}))(w-u_{\delta,\varepsilon})\right|\nonumber \\
    &=& o(1)+\left|\int_{\Omega^c}(h_\delta'(u_{\sigma,\delta,\varepsilon})-h_\delta'(u_{\delta,\varepsilon}))(w-u_{\delta,\varepsilon})\right|,
  \end{eqnarray*}
  then, if
  \begin{equation}\label{limit}
    \left|\int_{\Omega^c}(h_\delta'(u_{\sigma,\delta,\varepsilon})-h_\delta'(u_{\delta,\varepsilon}))(w-u_{\delta,\varepsilon})\right|\rightarrow0,
  \end{equation}
  as $\sigma\rightarrow0$, the result follows. Recalling \Cref{new proposition}, we remark that the integration in \eqref{limit} is, in fact, over a bounded domain. Furthermore, note that $(h_\delta'(u_{\sigma,\delta,\varepsilon})-h_\delta'(u_{\delta,\varepsilon}))$ is bounded, and $(u-w)\in L^2$, therefore, \eqref{limit} follows. 
\end{proof}
\begin{corollary}\label{c4.1}
	The function $u_{\delta,\varepsilon}$ satisfies (in the weak sense)
	\begin{equation}\label{4.3}	
    (-\Delta)^su_{\delta,\varepsilon}=f_\e'\left(\int_{\Omega^c}h_\delta(u_{\delta,\e}) \right)h_\delta'(u_{\delta,\e})\chi_{\Omega^c}
	\end{equation} 
	in $\{u_{\delta,\e}>\varphi\}$ and
	\begin{equation}\label{4.4}
	-C \left(C_\varphi
	+\frac{1}{\varepsilon\delta}\right)\le(-\Delta)^su_{\delta,\varepsilon} \le  f_\e'\left(\int_{\Omega^c}h_\delta(u_{\delta,\e})\right)h_\delta'(u_{\delta,\e})\chi_{\Omega^c}.  
	\end{equation}
\end{corollary}
\begin{proof}
    \Cref{r2.1} provides \eqref{4.3}. The second inequality of \eqref{4.4} follows from \eqref{4.1}; the first one from \eqref{3.1}, using \eqref{EL} and \Cref{c3.1}.
\end{proof}
To pass to the limit, as $\delta\to0$, we need uniform in $\delta$ estimates. Observe that since we are dealing with a non-local operator, we cannot expect higher regularity, as in \cite{Y16,TU17}. A way to bypass the issue would be using the flatness improvement technique, as in \cite{DT19,PT16}. Here we extrapolate the idea used in \cite{Y16} to the fractional framework, paving the way to the regularity of solutions.
\begin{theorem}\label{t4.1}
	There exists $C>0$ constant, depending only on $n$, $s$, but not on $\delta$, such that
	$$
	\|u_{\delta,\varepsilon}\|_{C^{0,s}}\leq C\left((1+\delta)C_\varphi+	\frac{1}{\varepsilon}\right).
	$$
\end{theorem}
\begin{proof}
	We consider three cases.\\
   
	\textsc{Case 1.} Let $u_{\delta,\varepsilon}(x_0)\le\delta$, and define
	$$
	v(y):=\frac{1}{\delta}u_{\delta,\varepsilon}\left(x_0+\delta^{\frac{1}{s}}y\right).
	$$
	Recalling 
	$$
	0<f_\varepsilon'\le\frac{1}{\varepsilon}\quad\text{ and }\quad 0\le h_\delta'\le\frac{1}{\delta},
	$$
    from \eqref{4.4} we obtain
	$$
	-C\left(\delta C_\varphi+\frac{1}{\varepsilon}\right)\le(-\Delta)^s v\le\frac{1}{\varepsilon}.
	$$
	On the other hand, $v(0)\le1$, therefore, 
	the Harnack inequality for the fractional Laplacian, \cite{CS07}, on compact subsets of $\R^n$ provides
	$$
	\|v\|_\infty\le C\left(\delta C_\varphi+\frac{1}{\varepsilon}+1\right),
	$$
	where the constant $C>0$ does not depend on $\delta$.  \Cref{prop:silv-regularity} (see also \Cref{r2.3}) then yields
	$$	\left[u_{\delta,\varepsilon}\right]_{C^{0,s}}=\left[ v\right]_{C^{0,s}}\le C\left(\delta C_\varphi+\frac{1}{\varepsilon}+1\right).
	$$
    \Cref{p2.1} then implies that
    $$    \|u_{\delta,\varepsilon}\|_{C^{0,s}}\le C\left((1+\delta) C_\varphi+\frac{1}{\varepsilon}\right).
    $$	
	\textsc{Case 2.} Let now $u_{\delta,\varepsilon}(x_0)=\varphi(x_0)$, and define
	$$
	v(y):=\frac{1}{\delta}\left[u_{\delta,\varepsilon}\left(x_0+\delta^{\frac{1}{s}}y\right)-\varphi\left(x_0+\delta^{\frac{1}{s}}y\right)\right].
	$$
	Once more, using \eqref{4.4}, we have
	$$
	-C\left(\delta C_\varphi+\frac{1}{\varepsilon}\right)\le(-\Delta)^s v\le\frac{1}{\varepsilon}+\delta C_\varphi.
	$$
	Since $v\ge0$ and $v(0)=0$, as above on the compact subsets of $\R^n$ one has
	$$
	\|v\|_\infty\le C\left(\delta C_\varphi+\frac{1}{\varepsilon}\right),
	$$
	where the constant $C>0$ does not depend on $\delta$. \Cref{prop:silv-regularity} now implies
	$$
	\|v\|_{C^{0,s}}\le C\left(\delta C_\varphi+\frac{1}{\varepsilon}\right),
	$$
	and therefore, again recalling \Cref{p2.1},
	$$
	\|u_{\delta,\varepsilon}\|_{C^{0,s}}\le C\left(\delta C_\varphi+\frac{1}{\varepsilon}\right)+C_\varphi.
	$$	
	\textsc{Case 3.} Finally, let $u_{\delta,\varepsilon}(x_0)>\max\{\varphi(x_0),\delta\}$. Set
	$$
	d:=\dist\big(x_0,\partial(\{u_{\delta,\varepsilon}>\varphi\}\cap\{u_{\delta,\varepsilon}>\delta\})\big)
	$$
	and let $z$ be a point where the distance is attained, i.e., $|z-x_0|=d$. Thus, either $u_{\delta,\varepsilon}(z)=\delta$ or $u_{\delta,\varepsilon}(z)=\varphi(z)$. Assume, for a moment, that $u_{\delta,\varepsilon}(z)=\delta$. Then the function
	$$
	v(y):=\frac{1}{d}\left(u_{\delta,\varepsilon}(x_0+d^{\frac{1}{s}}y)-\delta\right)
	$$
	is non-negative in $B_\rho$, where $\rho:=d^{1-\frac{1}{s}}$. It is also $s$-harmonic in $B_\rho$, as $h'_\delta=0$ in \eqref{4.3}. Additionally, $v(\tilde{z})=0$ for a point $\tilde{z}\in\partial B_\rho$. Then arguing as in Case 1, in $B_{\rho}$ we obtain
	$$
	\|v\|_\infty\le C\left(\delta C_\varphi+\frac{1}{\varepsilon}+1\right),
	$$
	and  \Cref{prop:silv-regularity} makes sure that
	$$	\left[u_{\delta,\varepsilon}\right]_{C^{0,s}}=\left[v\right]_{C^{0,s}}\le C\left(\delta C_\varphi+\frac{1}{\varepsilon}+1\right).
	$$
	If $u_{\delta,\varepsilon}(z)=\varphi(z)$, then the function
	$$
	w(y):=\frac{1}{d}\left(u_{\delta,\varepsilon}(x_0+d^{\frac{1}{s}}y)-\varphi(x_0+d^{\frac{1}{s}}y)\right)
	$$
	is non-negative and $s$-harmonic in $B_\rho$. Moreover, $w(y_0)=0$, for some $y_0\in\partial B_\rho$, where the estimate holds, as seen in Case 2. Arguing as above, we obtain the desired estimate also in this case. This completes the proof.	
\end{proof}
As a consequence of \Cref{p2.1}, the Arzel\`a-Ascoli Theorem, and \Cref{t4.1}, we obtain the next result.
\begin{corollary}\label{c4.2}
	If $u_{\sigma,\delta,\varepsilon}$ is a minimizer of $I_{\sigma,\delta,\varepsilon}$, then $u_{\sigma,\delta,\varepsilon}$ converges weakly (up to a subsequence as $\sigma,\delta\rightarrow0$) in $H^s(\R^n)$ to a function $u_\varepsilon$. This convergence is locally uniform. Moreover, there exists a constant $C>0$ such that
	$$
	\|u_\varepsilon\|_{C^{0,s}}\le C\left(C_\varphi+	\frac{1}{\varepsilon}\right).
	$$
\end{corollary}

\smallskip

\section{Back to the original problem}\label{s5}

Here we show that the function $u_\varepsilon$ from \Cref{c4.2} is a minimizer for a certain functional. This, in turn, provides information on the regularity of the exterior and interior free boundaries. Furthermore, we show that for $\varepsilon>0$ small enough (but fixed), the desired volume is attained automatically, which means that solutions of the penalized problems turn into solutions to our original problem inheriting all the properties.
\begin{theorem}\label{t5.1}
	The function $u_\varepsilon$ from \Cref{c4.2} is a minimizer of
	$$
	J_\varepsilon(u):=J(u)+
	f_\varepsilon(|\{u>0\}\setminus\Omega|)
	$$
	over the functions in $H^s(\R^n)$ that lie above $\varphi$. Here $J(u)$ is defined by \eqref{Jdef}.
\end{theorem}
\begin{proof}
	We argue by contradiction and assume that $\inf J_{\e}<J_{\e}(u_\e)$. Hence, for given $\theta>0$ there exists $v\in H^s(\R^n)$ with $v\geq\varphi$ such that $J_{\e}(v)<J_{\e}(u_\e)-2\theta$. Since $J(u_\varepsilon)$ and $J(v)$ are finite, then 
    \[ \int_{B_r^c}\int_{B_r^c}\frac{|u_\e(x)-u_\e(y)|^2-|v(x)-v(y)|^2}{|x-y|^{n+2s}}\,\mathrm{d}x\,\mathrm{d}y<\frac{\theta}{2},\]
	when $r>0$ is big enough. Also, we note that both $\{u_\varepsilon>0\}\setminus\Omega$ and $\{v>0\}\setminus\Omega$ have finite measure, since otherwise $f_\e$ would be infinity on the corresponding function. For $r>0$ sufficiently big,
	the sets $\{u_\varepsilon>0\}\cap B_r^c\setminus\Omega$ and $\{v>0\}\cap B_r^c\setminus\Omega$ have arbitrarily small volume. The continuity of $f_\e$ implies
	\[ |f_\e\left(|\{u_\varepsilon>0\}\cap B_r^c\setminus\Omega|\right)-f_\e\left(|\{v>0\}\cap B_r^c\setminus\Omega|\right)|<\frac{\theta}{2}. \]
	Therefore, 
	\begin{equation}\label{5.1}
	\begin{split}
	&\int_{B_r}\int_{B_{r}}\frac{|v(x)-v(y)|^{2}}{|x-y|^{n+2s}}\,\mathrm{d}x\,\mathrm{d}y+f_\varepsilon\left(\int_{\Omega^c\cap B_r}\chi_{\{v>0\}}\right) \\
	&<\int_{B_r}\int_{B_{r}}\frac{|u_\e(x)-u_\e(y)|^{2}}{|x-y|^{n+2s}}\,\mathrm{d}x\,\mathrm{d}y+f_\varepsilon\left(\int_{\Omega^c\cap B_r}\chi_{\{u_\varepsilon>0\}}\right)-\theta.
	\end{split}
	\end{equation}
	Since $h_\delta(v)\rightarrow\chi_{\{v>0\}}$, as $\delta\rightarrow0$, and $g_\sigma(v-\varphi)=0$, dominated convergence theorem yields
	\begin{equation}
	\begin{split}
	&\int_{B_r}\int_{B_{r}}\frac{|v(x)-v(y)|^{2}}{|x-y|^{n+2s}}\,\mathrm{d}x\,\mathrm{d}y+f_\varepsilon\left(\int_{\Omega^c\cap B_r}\chi_{\{v>0\}}\right) \\
	&=\int_{B_r}\int_{B_{r}}\frac{|v(x)-v(y)|^{2}}{|x-y|^{n+2s}}\,\mathrm{d}x\,\mathrm{d}y+\int_\Omega g_\sigma(v-\varphi) \\
    & \quad  +\lim_{\delta\rightarrow0}f_\varepsilon\left(\int_{\Omega^c\cap B_r}h_\delta(v)\right).
	\end{split}
	\end{equation}
	Note that if $\tau>0$ is small, then $h_\delta(u_{\delta,\varepsilon})=\chi_{\{u_{\delta,\varepsilon}>0\}}$ on $\{u_\varepsilon>\tau\}$ for $\delta$ small. Denoting $\Gamma:=\Omega^c\cap B_r\cap\{u_\varepsilon\geq\tau\}$, recalling that by definition $f_
    \varepsilon$ is a continuous monotone function, and using Fatou lemma, we estimate
	\begin{equation}\label{5.3}
	\begin{split}
	&\int_{B_r}\int_{B_{r}}\frac{|u_\e(x)-u_\e(y)|^{2}}{|x-y|^{n+2s}}\,\mathrm{d}x\,\mathrm{d}y+f_\varepsilon\left(\int_{\Omega^c\cap B_r}\chi_{\{u_\varepsilon>0\}}\right)-\theta\\
	&\leq\int_{B_r}\int_{B_{r}}\frac{|u_\e(x)-u_\e(y)|^{2}}{|x-y|^{n+2s}}\,\mathrm{d}x\,\mathrm{d}y+f_\varepsilon\left(\int_{\Gamma}\chi_{\{u_\varepsilon>0\}}\right)-\frac{\theta}{2}\\	&\leq\liminf_{\delta\rightarrow0}\left[\int_{B_r}\int_{B_{r}}\frac{|u_{\delta,\e}(x)-u_{\delta,\e}(y)|^{2}}{|x-y|^{n+2s}}\,\mathrm{d}x\,\mathrm{d}y+f_\varepsilon\left(\int_{\Gamma}\chi_{\{u_{\delta,\varepsilon}>0\}}\right)\right]-\frac{\theta}{2}\\
	&\le\liminf_{\delta\rightarrow0}\left[\int_{B_r}\int_{B_{r}}\frac{|u_{\delta,\e}(x)-u_{\delta,\e}(y)|^{2}}{|x-y|^{n+2s}}\,\mathrm{d}x\,\mathrm{d}y+f_\varepsilon\left(\int_{\Gamma}h_\delta(u_{\delta,\varepsilon})\right)\right]-\frac{\theta}{2}\\
	&\leq\liminf_{\sigma,\delta\rightarrow0}\left[\int_{B_r}\int_{B_{r}}\frac{|u_{\sigma,\delta,\e}(x)-u_{\sigma,\delta,\e}(y)|^{2}}{|x-y|^{n+2s}}\,\mathrm{d}x\,\mathrm{d}y+f_\varepsilon\left(\int_{\Gamma}h_\delta(u_{\sigma,\delta,\varepsilon})\right)\right]-\frac{\theta}{2}\\
	&\leq\liminf_{\sigma,\delta\rightarrow0}\left[\int_{B_r}\int_{B_{r}}\frac{|u_{\sigma,\delta,\e}(x)-u_{\sigma,\delta,\e}(y)|^{2}}{|x-y|^{n+2s}}\,\mathrm{d}x\,\mathrm{d}y+\int_\Omega g_\sigma(u_{\sigma,\delta,\varepsilon}-\varphi)\right.\\
	&\quad+\left.f_\varepsilon\left(\int_{\Omega^c\cap B_r}h_\delta(u_{\sigma,\delta,\varepsilon})\right)\right]-\frac{\theta}{2}.
	\end{split}
	\end{equation}
	From \eqref{5.1}-\eqref{5.3}, we obtain
	$$
	J_{\sigma,\delta,\varepsilon}(v)<J_{\sigma,\delta,\varepsilon}(u_{\sigma,\delta,\varepsilon}) -\frac{\theta}{4}<J_{\sigma,\delta,\varepsilon}(u_{\sigma,\delta,\varepsilon}), 
	$$
	which is a contradiction, since $u_{\sigma,\delta,\varepsilon}$ is a minimizer of $J_{\sigma,\delta,\e}$.
\end{proof}  
\begin{corollary}\label{c5.1}
	The Euler-Lagrange equation for $u_\e$ is 
	\[ \begin{cases}
	(-\Delta)^su_\e \ge 0\,\,\text{ in }\,\,\Omega,\\
	(-\Delta)^su_\e=0\,\,\text{ in }\,\,\Omega\cap\{u_\e>\varphi\},\\
	(-\Delta)^su_\e \le 0\,\,\text{ in }\,\,\Omega^c,\\
	(-\Delta)^su_\e=0\,\,\text{ in }\,\,\{u_\e>0\}\setminus\Omega.
	\end{cases} \]
	\end{corollary}
The previous theorem puts us in the framework of \cite{TT15}, where the authors analyze properties of minimizers of $J_\e$. Thus, one has the following.
\begin{theorem}\label{t5.2}
For $\varepsilon>0$ small, the function $u_\varepsilon$ from \Cref{c4.2}, solves the problem \eqref{P}. Moreover, $u_\e$ is H\"older continuous with exponent $s$, and that regularity is optimal.
\end{theorem}
\begin{proof}
	By \Cref{TT}, when $\e>0$ is small (but fixed), then $|\{u_\varepsilon>0\}\setminus\Omega|=\gamma$. The latter implies that (see \Cref{c5.1}) $u_\e\in\mathbb{K}$ and additionally that $f_\varepsilon(|\{u_\varepsilon>0\}\setminus\Omega|)=0$. Therefore, the function $u_\varepsilon$ solves \eqref{P}. In other words, for $\e>0$ small enough, we have the desired volume, and minimizers of $J_\e$ turn into minimizers of $J$, i.e., solutions of the original problem. The $s$-H\"older regularity of $u_\e$ is observed in \Cref{c4.2}, and it is optimal, \cite[Theorem 2.1]{TT15}.
\end{proof}
\begin{remark}
	Observe that any minimizer of \eqref{P} can be interpreted as a minimizer of $J_{\varepsilon}$, for $\varepsilon > 0$ small enough. Thus, all minimizers of \eqref{P} have the optimal $s$-H\"older regularity.
\end{remark}

\smallskip

\section{Regularity of the free boundaries}
\Cref{t5.2} implies non-degeneracy and positive density results for solutions, \cite[Lemma 2.2]{TT15} and \cite[Theorem 2.3]{TT15} respectively, as stated in the following lemma.
\begin{lemma}\label{t5.3}
	If $u$ is a solution of \eqref{P}, and $x_0\in\partial\{u>0\}\cap\Omega$, then there exists a constant $C>0$ such that 
	$$
	\sup_{B_r(x_0)}u\ge Cr^s,
	$$
	for $0<r<\frac{1}{2}\text{dist}(x_0,\partial\Omega)$. Furthermore, there exists a constant $c>0$ such that 
	$$
	|\{u=0\}\cap B_r(x_0)|\ge cr^n\,\,\,\textrm{ and }\,\,\,|\{u>0\}\cap B_r(x_0)|\ge cr^n.
	$$
\end{lemma}
The regularity of the interior free boundary comes from that of the fractional obstacle problem. 
    \begin{theorem}\label{t6.2}
        If $u$ is a minimizer of \eqref{P}, and $\varphi\in C^{2,1}(\R^n)$, then the interior free boundary $\partial\{u>\varphi\}$ is locally a $C^{1,\tau}$ surface for some $\tau\in(0,1)$.
    \end{theorem}
    \begin{proof}
        Indeed, since $\partial\{u>0\}\subset\Omega^c$, then minimizer $v$ of the fractional energy $J$ over $\{v\ge\varphi\}$ with zero data on $\{u>0\}^c$ is the solution of the fractional obstacle problem in $\{u>0\}$. Moreover, the uniqueness of the solution of the obstacle problem implies that $v\equiv u$. Thus, $u$ is the solution of the fractional obstacle problem, and therefore, as $\varphi\in C^{2,1}(\R^n)$, the free boundary $\partial\{u>\varphi\}$ is locally a $C^{1,\tau}$ surface for some $\tau\in(0,1)$, \cite[Theorem 7.7]{CSS08}.
    \end{proof}
The next result concerns the regularity of the exterior free boundary.
\begin{theorem}\label{t5.5}
	If $u$ is a minimizer of \eqref{P}, then
	\begin{itemize}
		\item $\mathcal{H}^{n-1}(\mathcal{K}\cap\partial\{u>0\}\cap\mathbb{R}^n)<\infty$, for every compact set $\mathcal{K}\subset\Omega$.
		\item The reduced free boundary $\partial^*\{u>0\}\cap\mathbb{R}^n$ is locally a $C^{1,\beta}$ surface.
	\end{itemize}
\end{theorem}
\begin{proof}
	This follows from \Cref{t5.1} and \Cref{TTfb}.
\end{proof}
\begin{remark}\label{r5.1}
	As in \cite[Theorem 6.4]{Y16} (see also \cite[Lemma 6.2]{TU17}), the positivity set is well localized in a bounded set, meaning that the optimization is in fact in a big (but bounded) domain rather than the whole space (see \Cref{new proposition}).
\end{remark}

\vspace{0.25cm}

\noindent{\bf Acknowledgments.} DM was partially supported by CNPq-Brazil through grant 311354/2019-0. RT was supported by the King Abdullah University of Science and Technology (KAUST). The authors thank the anonymous referee for insightful comments and constructive suggestions that helped to revise and enhance the quality of the manuscript.

\end{document}